\documentclass[12pt, reqno]{amsart}
\usepackage{amsmath, amsthm, amscd, amsfonts, amssymb, graphicx}
\usepackage[bookmarksnumbered, colorlinks, plainpages]{hyperref}
\usepackage{mathrsfs}
\newtheorem{theorem}{Theorem}[section]

\newtheorem{proposition}[theorem]{Proposition}
\newtheorem{corollary}[theorem]{Corollary}
\newtheorem{remark}[theorem]{Remark}
\theoremstyle{definition}

\numberwithin{equation}{section}

\begin{document}
\title[An Alternative Estimate for $\cdots $]{An Alternative Estimate for the Numerical Radius of Hilbert Space Operators}
\author[M. Shah Hosseini, B. Moosavi \& H.R. Moradi]{Mohsen Shah hosseini$^1$, Baharak moosavi$^2$ and Hamid Reza Moradi$^3$}
\address{$^1$Department of Mathematics, 
Shahr-e-Qods Branch, Islamic Azad University,
Tehran, Iran.} 
\email{mohsen${\_}$shahhosseini@yahoo.com}
\address{$^2$Department  of Mathematics, Safadasht Branch, Islamic Azad University, Tehran,
Iran.}
\email{baharak${\_}$moosavie@yahoo.com}
\address{$^3$Young Researchers and Elite Club, Mashhad Branch, Islamic Azad University, Mashhad, Iran.}
\email{hrmoradi@mshdiau.ac.ir}
 \maketitle
\begin{abstract}
We give an alternative lower bound for the numerical radii of Hilbert space operators. As a by-product, we find conditions such that 
        \begin{equation*}
\omega\left(\left[\begin{array}{cc} 
0 & R \\
S & 0
\end{array}\right]\right)=\frac{\Vert R \Vert +\Vert S\Vert }{2}
\end{equation*}
where $R, S \in \mathbb{B}(\mathcal{H})$.
\end{abstract}
\section{Introduction and summary}
Let $\mathbb{B}(\mathcal{H})$ denote the ${{C}^{*}}$-algebra of all bounded linear operators on a complex Hilbert
space $\mathcal{H}$ with inner product $\langle\cdot, \cdot\rangle$ and let $ \mathbb{B}^{-1}(\mathcal{H}) $ denote the set of all invertible
operators in $\mathbb{B}(\mathcal{H})$.
 For $T \in \mathbb{B}(\mathcal{H})$, let
\begin{equation*}
\omega \left( T \right)=\sup \left\{ \left| \left\langle Tx,x \right\rangle  \right|:\text{ }x\in \mathcal{H},\left\| x \right\|=1 \right\}
\end{equation*}
and 
\begin{equation*}
\left\| T \right\|=\sup \left\{ \left\| Tx \right\|:\text{ }x\in \mathcal{H},\left\| x \right\|=1 \right\},
\end{equation*}
respectively, denote the numerical radius and operator norm of $T$. Recall that, for all $T \in \mathbb{B}(\mathcal{H})$,
\begin{eqnarray}\label{1.1}
\frac{\Vert T \Vert}{2}\leq \omega(T) \leq \Vert T \Vert.
\end{eqnarray}
For more information and background, we refer to the book by Gustafson and Rao \cite{[3]}. 

Berger \cite{[1]} showed that for any $ T \in \mathbb{B}(\mathcal{H}) $ and natural number $n$,
\begin{equation*} 
\omega( T^n ) \leq  \omega^n(T). 
\end{equation*}
In particular, if $T$ is a normal operator, then
\begin{equation}\label{normal} 
\omega ( T ) = \Vert T\Vert 
\end{equation}
 and 
\begin{equation*} 
\Vert T^ n \Vert = \Vert T \Vert^ n. 
\end{equation*}

Several numerical radius inequalities improving \eqref{1.1} have been recently given in
\cite{[7],[9],[8],1}. 

The following inequality due to Holbrook \cite{[4]} asserts that
\begin{equation}\label{zarb}
\omega( RS ) \leq 4 \omega(R) \omega(S)
\end{equation}
for any $ R,S\in \mathbb{B}(\mathcal{H}) $. 
In the same paper, the author also proved  if  $ RS = SR$, then  
\begin{equation*}
\omega( RS ) \leq 2 \omega(R) \omega(S).
\end{equation*}
 If $R$ and $S$ are operators in $\mathbb{B}(\mathcal{H})$, we write the direct sum $R\oplus S$
for the $2\times2$ operator matrix $\left[\begin{array}{cc}
R & 0 \\
0 & S
\end{array}\right]$, regarded as an operator on $\mathcal{H}\oplus \mathcal{H}$. Thus
\begin{equation*}
 \omega(R\oplus S) =\max(\omega(R) ,\omega(S)).
\end{equation*}
In addition,
\begin{equation}\label{1.5} \Vert R\oplus S \Vert =\left  \Vert \left[\begin{array}{cc}
0 & R  \\
S & 0
\end{array}\right] \right \Vert =\max(\Vert R\Vert ,\Vert S \Vert).
\end{equation}
It is shown in \cite{[18]} that if $ R, S \in \mathbb{B}(\mathcal{H}) $, then
\begin{equation}\label{farei}
\sqrt[2n]{\max(\omega((RS)^n), \omega((SR)^n)}
\leq \omega\left(\left[\begin{array}{cc} 
0 & R \\
S & 0
\end{array}\right]\right)\leq \frac{\Vert R\Vert +\Vert S \Vert}{2}
\end{equation}
for $ n=1, 2,\ldots$.

In this paper, we first prove an alternative estimate for the LHS of \eqref{1.1}. As an application, we improve the inequality \eqref{zarb}. Additionally, we will provide conditions under which the RHS of the inequality \eqref{farei} will change to equality. Our result determines the numerical radius of real  off-diagonal $2\times 2$ matrices. 
\section{Main Results}
Let $T={{T}_{1}}+i{{T}_{2}}$ be the Cartesian decomposition of $T$, where ${{T}_{1}}=\operatorname{Re}T=\frac{T+{{T}^{*}}}{2}$ and ${{T}_{2}}=\operatorname{Im}T=\frac{T-{{T}^{*}}}{2i}$. Then ${{T}_{1}}$ and ${{T}_{2}}$ are self-adjoint. So, $\min \left( \frac{{{\left\| T-{{T}^{*}} \right\|}^{2}}}{2},\frac{{{\left\| T+{{T}^{*}} \right\|}^{2}}}{2} \right)=2\min \left(   \left\| {{T}_{1}} \right\|^2,\left\| {{T}_{2}} \right\|^2 \right)$.\\ For the sake of convenience, we prepare the following notations: 
\[D\left( T \right)=2\min \left( {{\left\| {{T}_{1}} \right\|}^{2}},{{\left\| {{T}_{2}} \right\|}^{2}} \right)~\text{ and }~\alpha \left( T \right)=\underset{\left\| x \right\|=1}{\mathop{\inf }}\,{{\left\| Tx \right\|}^{2}}.\]

The main result of the paper reads as follows.
\begin{theorem}\label{th.asli}
Let $T \in \mathbb{B}(\mathcal{H})$ with the Cartesian decomposition $T={{T}_{1}}+i{{T}_{2}}$. Then 
\begin{equation}\label{asli}
{{\left\| T \right\|}^{2}}+\max \left( \alpha \left( T \right),\alpha \left( {{T}^{*}} \right) \right)\le 2{{\omega }^{2}}\left( T \right)+D\left( T \right).
\end{equation}
\end{theorem}
\begin{proof}
We use the following identity
\begin{equation}\label{motevazi}
\Vert a\Vert ^2 + \Vert b \Vert ^2=\frac{\Vert a-b\Vert ^2+\Vert a+b\Vert ^2}{2}\  \ ( a, b\in \mathcal{H} ).
\end{equation}
 Suppose that $ x \in \mathcal{H} $ with $ \Vert x\Vert= 1 $. Choose 
$ a =T x, b=T^*x $ in \eqref{motevazi} to give
\begin{equation*}
\Vert T x\Vert ^2 + \Vert T^*x \Vert ^2=\frac{1}{2}(\Vert T x-T^*x\Vert ^2+\Vert T x+T^*x\Vert ^2).
\end{equation*}
On account of the definition of $ \alpha(T) $, we infer that
\begin{equation*}
\alpha(T) + \Vert T^*x \Vert ^2\leq \frac{1}{2}(\Vert T x-T^*x\Vert ^2+\Vert T x+T^*x\Vert ^2).
\end{equation*}
By taking the supremum over $x\in \mathcal{H}$ with $ \Vert x\Vert= 1 $, it follows that
\begin{equation*}
\alpha(T) + \Vert T\Vert ^2\leq \frac{1}{2}(\Vert T -T^*\Vert ^2+\Vert T +T^*\Vert ^2).
\end{equation*}
Since $ (T +T^*)$ is normal, \eqref{normal} yields
\begin{equation*}
\Vert T +T^*\Vert \leq 2\omega(T).
\end{equation*}
Therefore,
\begin{equation*}
\alpha(T) + \Vert T\Vert ^2\leq 2 \omega ^2(T)+ \frac{\Vert T -T^*\Vert ^2}{2}.
\end{equation*}
Similarly,
\begin{equation*}
\alpha(T^*) + \Vert T\Vert ^2\leq 2 \omega ^2(T)+ \frac{\Vert T -T^*\Vert ^2}{2}
\end{equation*}
and so 
\begin{equation}\label{shomareh}
{{\left\| T \right\|}^{2}}+\max \left( \alpha \left( T \right),\alpha \left( {{T}^{*}} \right) \right)\le 2{{\omega }^{2}}\left( T \right)+\frac{{{\left\| T-{{T}^{*}} \right\|}^{2}}}{2}.
\end{equation}
Replacing $ T $ by $ iT $, we have 
\begin{equation}\label{shomareh01}
{{\left\| T \right\|}^{2}}+\max \left( \alpha \left( T \right),\alpha \left( {{T}^{*}} \right) \right)\le 2{{\omega }^{2}}\left( T \right)+\frac{{{\left\| T+{{T}^{*}} \right\|}^{2}}}{2}.
\end{equation}
Now, from \eqref{shomareh} and \eqref{shomareh01} we get \eqref{asli}, as required. 
\end{proof}

For an operator $A\in \mathcal{B}\left( \mathcal{H} \right)$, if $\mathcal{H}$ is infinite-dimensional, then ${{\inf }_{\left\| x \right\|=1}}{{\left\| Tx \right\|}^{2}}$ and ${{\inf }_{\left\| x \right\|=1}}{{\left\| {{T}^{*}}x \right\|}^{2}}$ may be different (consider for example the unilateral shift operator). If $\mathcal{H}$ is finite-dimensional, then ${{\inf }_{\left\| x \right\|=1}}{{\left\| Tx \right\|}^{2}}={{\inf }_{\left\| x \right\|=1}}{{\left\| {{T}^{*}}x \right\|}^{2}}$. In this case we can write \eqref{asli} in the following form
\[{{\left\| T \right\|}^{2}}+\underset{\left\| x \right\|=1}{\mathop{\inf }}\,{{\left\| Tx \right\|}^{2}}\le 2{{\omega }^{2}}\left( T \right)+D\left( T \right).\]
It is also interesting to note that if $T$ is invertible, then regardless of the dimension of $\mathcal{H}$, ${{\inf }_{\left\| x \right\|=1}}{{\left\| Tx \right\|}^{2}}={{\inf }_{\left\| x \right\|=1}}{{\left\| {{T}^{*}}x \right\|}^{2}}={{\left\| {{T}^{-1}} \right\|}^{-2}}$.

In the next result we improve the LHS of \eqref{1.1}, thanks to Theorem \ref{asli}.
\begin{proposition}\label{prop}
Let  $T \in \mathbb{B}(\mathcal{H})$ with the Cartesian decomposition $T={{T}_{1}}+i{{T}_{2}}$. Then 
\begin{equation}\label{2}
		\left\| T \right\|\le \sqrt{2{{\omega }^{2}}\left( T \right)-\max \left( \alpha \left( T \right),\alpha \left( {{T}^{*}} \right) \right)+D\left( T \right)}\le 2\omega \left( T \right).
\end{equation}
\end{proposition}
\begin{proof}
First of all, we note that
\[{{\omega }^{2}}\left( T \right)=\underset{\left\| x \right\|=1}{\mathop{\sup }}\,{{\left| \left\langle Tx,x \right\rangle  \right|}^{2}}=\underset{\left\| x \right\|=1}{\mathop{\sup }}\,\left( {{\left\langle T_1x,x \right\rangle }^{2}}+{{\left\langle T_2x,x \right\rangle }^{2}} \right).\]
On the other hand,		
\[\begin{aligned}
& \underset{\left\| x \right\|=1}{\mathop{\sup }}\,\left( {{\left\langle {{T}_{1}}x,x \right\rangle }^{2}}+{{\left\langle {{T}_{2}}x,x \right\rangle }^{2}} \right) \\ 
& \ge \underset{\left\| x \right\|=1}{\mathop{\sup }}\,{{\left\langle {{T}_{1}}x,x \right\rangle }^{2}}\left( \text{resp}.\text{ }\ge \underset{\left\| x \right\|=1}{\mathop{\sup }}\,{{\left\langle {{T}_{2}}x,x \right\rangle }^{2}} \right) \\ 
& \ge {{\left\| {{T}_{1}} \right\|}^{2}}\left( \text{resp}.\text{ }\ge {{\left\| {{T}_{2}} \right\|}^{2}} \right) \\ 
& \ge \min \left( {{\left\| {{T}_{1}} \right\|}^{2}},{{\left\| {{T}_{2}} \right\|}^{2}} \right). 
\end{aligned}\]
Consequently,
\begin{equation}\label{3}
D\left( T \right)\le 2{{\omega }^{2}}\left( T \right).
\end{equation}
Combining \eqref{asli} and \eqref{3} we get \eqref{2}, as required.
\end{proof}
Recently in \cite{2}, the authors tried to show
		\begin{equation}\label{1}
		\left\| A \right\|\le \sqrt{2}\omega \left( A \right)	
		\end{equation}
		holds, whenever $A$ is invertible operator. Cain \cite{[2]} by giving a counterexample showed inequality \eqref{1} does not hold, even for invertible operators.
		In the next result, we provide some conditions under which \eqref{1} can be true.
\begin{corollary}
	Let $R \in \mathbb{B}^{-1}(\mathcal{H})$. If $D\left( R \right)\le {{\left\| {{R}^{-1}} \right\|}^{-2}}$,  then
	\begin{equation}\label{radical2}
	\Vert R\Vert \leq \sqrt{ 2} \omega(R).
	\end{equation}
If in addition, $S \in \mathbb{B}^{-1}(\mathcal{H})$ with $D\left( S \right)\le {{\left\| {{S}^{-1}} \right\|}^{-2}}$, then
	\begin{equation}\label{omega zarb 2}
	\omega(RS)\leq 2\omega(R)\omega(S).
	\end{equation}
	\begin{proof}
Replace $T$ in \eqref{asli} with $R$. Since $R$ is invertible the remarks just above Proposition \ref{prop} show that the``max" term in the resulting inequality is equal to $h=\left\| {{R}^{-1}} \right\|^{-2}$. Since we assumed that $h$ dominates $D(T)$ a new inequality can be obtained replacing  $D(T)$ with $h$. To get \eqref{radical2} subtract $h$ from both sides of this new inequality.

For	the inequality \eqref{omega zarb 2}, we can write
		\[\omega \left( RS \right)\le \left\| RS \right\|\le \left\| R \right\|\left\| S \right\|\le 2\omega \left( R \right)\omega \left( S \right)\]
where in the first inequality we used the RHS of \eqref{1.1}, the second inequality follows from the sub-multiplicative property of operator norm, and the third inequality obtains from \eqref{radical2}. 
	\end{proof}
\end{corollary}



For convenience, we use the following notation in Corollary \ref{tala}: $$g(T)={{\left\| T \right\|}^{2}}+\max \left( \alpha \left( T \right),\alpha \left( {{T}^{*}} \right) \right)-D\left( T \right).$$
\begin{corollary}\label{tala}
Let $R, S \in \mathbb{B}(\mathcal{H})$ and $ T=\left[\begin{array}{cc} 
0 & R \\
S & 0
\end{array}\right] $. If  ${\left( \left\| R \right\|+\left\| S \right\| \right)}^{2} \le 2 g(T)$, then

\begin{equation*}
\omega(T)=\frac{\Vert R \Vert +\Vert S\Vert }{2}.
\end{equation*}
\end{corollary}
\begin{proof}
Clearly ${\left( \left\| R \right\|+\left\| S \right\| \right)}^{2}\leq 4 {{\omega }^{2}}\left( T \right)$, so taking the square root gives  inequality \eqref{farei} reversed.
\end{proof}
\begin{remark}
Let $R, S \in \mathbb{B}(\mathcal{H})$ and $ T=\left[\begin{array}{cc} 
0 & R \\
S & 0
\end{array}\right] $. For this special $T$ we have 
\begin{equation*}
g(T)=\max \left( {{\left\| R \right\|}^{2}},{{\left\| S \right\|}^{2}} \right)+\max \left( \alpha \left( T \right),\alpha \left( {{T}^{*}} \right) \right)-\min \left( \frac{{{\left\| R-{{S}^{*}} \right\|}^{2}}}{2},\frac{{{\left\| R+{{S}^{*}} \right\|}^{2}}}{2} \right),
\end{equation*}
 and if in addition $R$ and $S$ are invertible 
 \begin{equation*}
g(T)=\max \left( {{\left\| R \right\|}^{2}},{{\left\| S \right\|}^{2}} \right)+\min \left({{\left\|{{R}^{-1}}\right\|}^{-2}},{{\left\|{{S}^{-1}}\right\|}^{-2}} \right)-\min\left( \frac{{{\left\| R-{{S}^{*}} \right\|}^{2}}}{2},\frac{{{\left\| R+{{S}^{*}} \right\|}^{2}}}{2} \right)
\end{equation*}
   because the ``max" term becomes $\Vert T^{-1}\Vert ^{-2}=\min \left( {{\left\| {{R}^{-1}} \right\|}^{-2}},{{\left\| {{S}^{-1}} \right\|}^{-2}} \right)$. If $R$ and $S$ are the scalar matrices $rI$ and $sI$ with $r$ and $s$ complex numbers, the inequality assumed valid in Corollary \ref{tala}  takes the form $\min \left( {{\left| r-\overline{s} \right|}^{2}},{{\left| r+\overline{s} \right|}^{2}} \right)\le {{\left( \left| r \right|-\left| s \right| \right)}^{2}}$ (and this is always true if  $r$ and $s$  are real). That is one way to conclude that $\omega(T)=\frac{\Vert R \Vert +\Vert S\Vert }{2}$.  Another way is by computing it from the definition, which is not entirely straight forward.

\end{remark}
\section*{Acknowledgement}
The authors would like to thank an anonymous referee for pointing out a crucial mistake in Theorem \ref{th.asli}.



\end{document}